\documentclass{amsart}

\newtheorem{theorem}{Theorem}[section]
\newtheorem{lemma}[theorem]{Lemma}

\theoremstyle{definition}
\newtheorem{definition}[theorem]{Definition}
\newtheorem{example}[theorem]{Example}

\newtheorem{prop}[theorem]{Proposition}
\newtheorem{cor}[theorem]{Corollary}

\theoremstyle{remark}
\newtheorem{remark}[theorem]{Remark}

\numberwithin{equation}{section}

\begin{document}

\title{On distribution and almost convergence of bounded sequences}

\author{Chao You}
\address{Department of Mathematics $\&$ The Academy of Fundamental and Interdisciplinary Science\\Harbin Institute of Technology\\
Harbin 150001, Heilongjiang, People's Republic of China}
\email{hityou1982@gmail.com}

\author{Bao Qi Feng}
\address{Department of Mathematical Sciences\\Kent State University, Tuscarawas\\New Philadelphia, OH 44663, USA}
\email{bfeng@kent.edu}

\subjclass[2000]{Primary 40G05, 46A35, 54A20; Secondary 11K36}



\keywords{Banach limit, almost convergent sequence, properly
distributed sequence, simply distributed sequence}

\begin{abstract}
In this paper, we give the concepts of properly distributed and
simply distributed sequences, and prove that they are almost
convergent. Basing on these, we review the work of Feng and Li
[Feng, B. Q. and Li, J. L., Some estimations of Banach limits, J.
Math. Anal. Appl. 323(2006) No. 1, 481-496. MR2262220 46B45
(46A45).], which is shown to be a special case of our generalized
theory.
\end{abstract}

\maketitle

\section{preliminary and background}
Let $l^{\infty}$ be the Banach space of bounded sequences of real
numbers $x:=\{x(n)\}_{n=1}^{\infty}$ with norm
$\|x\|_{\infty}=\sup|x(n)|$. As an application of Hahn-Banach
theorem, a \emph{Banach limit} $L$ is a bounded linear functional on
$l^{\infty}$, which satisfies the following properties:

(a)If $x:=\{x(n)\}_{n=1}^{\infty}\in l^{\infty}$ and $x(n)\geq 0$,
then $L(x)\geq 0$;

(b)If $x:=\{x(n)\}_{n=1}^{\infty}\in l^{\infty}$ and
$Tx=\{x(2),x(3),\ldots\}$, then $L(x)=L(Tx)$, where $T$ is the
\emph{translation operator};

(c)$L(1)=1$, where $1:=\{1,1,\ldots\}$;

(d)$\|L\|=1$;

(e)If $x:=\{x(n)\}_{n=1}^{\infty}\in c$, then
$L(x)=\lim_{n\rightarrow\infty}x(n)$, where $c$ is the Banach
subspace of $l^{\infty}$ consisting of convergent sequences.

Since the Hahn-Banach norm-preserving extension is not unique, there
must be many Banach limits in the dual space of $l^{\infty}$, and
usually different Banach limits have different values at the same
element in $l^{\infty}$. However, there indeed exist sequences whose
values of all Banach limits are the same. Condition (e) is a trivial
example. Besides that, there also exist nonconvergent sequences
satisfying this property, for such examples please see \cite{Feng1}
and \cite{Feng2}. In \cite{Lorentz}, G. G. Lorentz called a sequence
$x:=\{x(n)\}_{n=1}^{\infty}$ \emph{almost convergent}, if all Banach
limits of $x$, $L(x)$, are the same. In his paper, Lorentz proved
the following criterion for almost convergent sequences:

\begin{theorem}
A sequence $x:=\{x(n)\}_{n=1}^{\infty}\in l^{\infty}$ is almost
convergent if and only if
$$
\lim_{n\rightarrow\infty}\frac{1}{n}\sum_{t=i}^{i+n-1}x(t)=L(x)
$$
uniformly in $i$.
\end{theorem}

There is no doubt that Lorentz' theorem is a landmark in Banach
limit theory, which in theory points out all the almost convergent
sequences. Recently, basing on Lorentz \cite{Lorentz} and Sucheston
\cite{Sucheston}, Feng B. Q. and Li J. L. gave another way
\cite{Feng1} to find the value of Banach limits of $x$, where $x$ is
an element of the space of almost convergent sequences with some
properties. In this paper, we will make a remark on the concept of
essential subsequence (Definition 2, \cite{Feng1}), then cite
Theorem 4(\cite{Feng1}) to develop our theory, and at last use our
theory to review two main results in \cite{Feng1}, in the bid to
include \cite{Feng1} into our framework and show that we have
genuinely done a work of generalization in theory. Thus, we'd better
make a short introduction to the main results of \cite{Feng1} first,
making the notations and terminologies available.

\begin{definition}[Definition 1, \cite{Feng1}]
A real number $a$ is said to be a sub-limit of the sequence
$x:=\{x(n)\}_{n=1}^{\infty}\in l^{\infty}$, if there exists a
subsequence $\{x(n_k)\}_{k=1}^{\infty}$ of $x$ with limit $a$. The
set of all sub-limits of $x$ is denoted by $S(x)$ and the set of all
limit points of $S(x)$ is denoted by $S'(x)$.
\end{definition}

\begin{definition}[Definition 3, \cite{Feng1}]\label{weight}
Let $x:=\{x(n)\}_{n=1}^{\infty}\in l^{\infty}$, and let
$\{x(n_k)\}_{k=1}^{\infty}$ be a subsequence of $x$. Define
$$
w^u(\{x(n_k)\})=\limsup_{n\rightarrow\infty}\left(\sup_i\frac{A(\{k:i\leq
n_k\leq i+n-1\})}{n}\right)
$$
and
$$
w_l(\{x(n_k)\})=\liminf_{n\rightarrow\infty}\left(\inf_i\frac{A(\{k:i\leq
n_k\leq i+n-1\})}{n}\right),
$$
where $A(E)$ is the cardinality of the set $E$. $w^u(\{x(n_k)\})$
and $w_l(\{x(n_k)\})$ are called the upper and lower weights of the
subsequence $\{x(n_k)\}_{k=1}^{\infty}$ respectively. If
$w^u(\{x(n_k)\})=w_l(\{x(n_k)\})$, then the subsequence
$\{x(n_k)\}_{k=1}^{\infty}$ is said to be weightable and the weight
of $\{x(n_k)\}_{k=1}^{\infty}$ is denoted by $w(\{x(n_k)\})$, and
$w(\{x(n_k)\})=w^u(\{x(n_k)\})=w_l(\{x(n_k)\})$.
\end{definition}

\begin{remark}
It should be emphasized that our Definition \ref{weight} is slightly
different from Definition 3(\cite{Feng1}), with
$\lim_{n\rightarrow\infty}$ there replaced by
$\limsup_{n\rightarrow\infty}$ and $\liminf_{n\rightarrow\infty}$
for $w^u(\cdot)$ and $w_l(\cdot)$ respectively. Such expression is
more accurate, since there is no reason to guarantee the existence
of $\lim_{n\rightarrow\infty}$.
\end{remark}

\begin{definition}[Definition 2, \cite{Feng1}]
Suppose $a\in S(x)$ for some $x:=\{x(n)\}_{n=1}^{\infty}\in
l^{\infty}$. A subsequence $\{x(n_k)\}_{k=1}^{\infty}$ of $x$ is
called an essential subsequence of $a$ if it converges to $a$, and
for any subsequence $\{x(m_t)\}_{t=1}^{\infty}$ of $x$ with limit
$a$, except finite entries, all its entries are entries of
$\{x(n_k)\}_{k=1}^{\infty}$.
\end{definition}

\begin{theorem}[Theorem 1, \cite{Feng1}]\label{essential subsequence}
Let $x:=\{x(n)\}_{n=1}^{\infty}\in l^{\infty}$. Suppose $a\in S(x)$.
Let $\{x(n_k)\}_{k=1}^{\infty}$ and $\{x(m_t)\}_{t=1}^{\infty}$ be
two essential subsequences of $a$. Then\\
$w^u(\{x(n_k)\})=w^u(\{x(m_t)\})$ and
$w_l(\{x(n_k)\})=w_l(\{x(m_t)\})$.
\end{theorem}

Theorem \ref{essential subsequence} points out that, for $a\in
S(x)$, all essential subsequences of $a$ have the same upper weight
and lower weight, respectively. They are called the \emph{upper} and
\emph{lower weights} of $a$ in the sequence $x$, and denoted by
$w^u(a)$ and $w_l(a)$, respectively. The \emph{weight} of $a$ in the
sequence $x$ is denoted by $w(a)$, if $w^u(a)=w_l(a)$.

We remark that not every sub-limit $a\in S(x)$ has an essential
subsequence. The following proposition shows that this happens only
when $a$ is an isolated sub-limit of $x$. This is an important
correction to \cite{Feng1}, and consideration on this problem
directly leads to our present work.

\begin{prop}\label{correction}\footnote{Special thanks goes to
Prof. J. L. Li for discussion with him on this proposition. In fact,
it was him that first pointed out this proposition and provided a
proof for the sufficient condition.} Let
$x:=\{x(n)\}_{n=1}^{\infty}\in l^{\infty}$ and suppose $a\in S(x)$.
$a$ has an essential subsequence if and only if $a$ is an isolated
sub-limit of $x$.
\end{prop}

\begin{proof}
If $a$ is an isolated sub-limit of $x$, then there exists
$\varepsilon_0>0$ such that
$(a-\varepsilon_0,a+\varepsilon_0)\bigcap S(x)=\{a\}$. Let
$\{x(n_k)\}$ denote all the terms of $x$ that lying in
$(a-\varepsilon_0,a+\varepsilon_0)$, we will show that $\{x(n_k)\}$
is the desired essential subsequence of $x$. Since $\{x(n_k)\}$ is
infinite and bounded, it must have at least one convergent
subsequence or sub-limit. But $a$ is an isolated sub-limit, hence
$\{x(n_k)\}$ has just one sub-limit, i.e., $a$. That's to say
$\{x(n_k)\}$ is convergent to $a$. For any subsequence $\{x(m_t)\}$
of $x$ that converging to $a$, from the definition of $\{x(n_k)\}$
and convergence of $\{x(m_t)\}$ to $a$, all of the terms of this
subsequence under consideration, except finite number of them, must
be in $\{x(n_k)\}$. So $\{x(n_k)\}$ is an essential subsequence of
$a$.

Conversely, suppose that $a$ has an essential subsequence
$\{x(n_k)\}$. Assume $a$ is not an isolated sub-limit of sequence
$x$, then there exist a sequence of sub-limits $\{a_n\}$ that
converges to $a$. We know, for each $a_n$ from $\{a_n\}$, there is a
subsequence $\{x_n^i\}$ that converges to $a_n$ when $i\rightarrow
\infty$. Without loss of generality, we can assume
$0<d_n=|a-a_n|<1/n$. Then, for each $n$, we can find $y_n$ from
$\{x_n^i\}$ such that $y_n$ doesn't lie in $\{x(n_k)\}$ and
$|y_n-a_n|<1/n$. Actually, this construction is possible. Since $a$
and $a_n$ are distinct with distance $d_n$, then we can find
positive integer $N_1$ and $N_2$ such that, when $k>N_1$, $i>N_2$,
it holds that $|x(n_k)-a|<d_n/3$ and $|x_n^i-a_n|<d_n/3$,
respectively. It is easy to see such $y_n$ can be found and
satisfying $|y_n-a|<d_n<1/n$. Here we have constructed a subsequence
$\{y_n\}$ converging to $a$, but not lying in the essential
subsequence $\{x(n_k)\}$, which leads to a contradiction.
\end{proof}

\begin{remark}
Since in \cite{Feng1} they just considered sequences with isolated
sub-limits, or a little complex case with only one limit point, this
ambiguous treatment of essential subsequences didn't lead to serious
mistakes.
\end{remark}

The following theorem is the most important result of \cite{Feng1},
which will be cited and reviewed later.

\begin{theorem}[Theorem 4, \cite{Feng1}]
Suppose $x:=\{x(n)\}_{n=1}^{\infty}\in l^{\infty}$ and\\
$S(x)=\{a_1,a_2,\ldots,a_m\}$ is a finite set, where $a_i\neq a_j$
if $i\neq j$. Then
\begin{align*}
\sum_{0<a_j\in S(x)}a_jw_l(a_j)+\sum_{0>a_j\in S(x)}a_jw^u(a_j)&\leq
L(x)\\
&\leq \sum_{0<a_j\in S(x)}a_jw^u(a_j)+\sum_{0>a_j\in
S(x)}a_jw_l(a_j).
\end{align*}
If $w(a_j)$ exists for each $j$, then $x$ is almost convergent and
for any Banach limit $L$, $L(x)=\sum_{j=1}^ma_jw(a_j)$.
\end{theorem}

This form of $L(x)=\sum_{j=1}^ma_jw(a_j)$ is much like the
\emph{integration sum} in measure and integration theory, so we ask
the question whether the unique Banach limit value of almost
convergent sequence could be expressed as an integral form? Previous
work shows this is related to the distribution of values appearing
in the sequence. In \cite{Kuipers}, the concept of \emph{uniform
distribution of sequences} was introduced as following: Suppose $x
\in l^{\infty}$ is a $[0,1]$-valued sequence, i.e. $0\leq x(n)\leq
1$ for each $n\in \mathbb{N}$. $x$ is called \emph{uniformly
distributed} if for any $[a,b)\subseteq [0,1]$,
$$\lim_{N\rightarrow\infty}\frac{A(\{n\in\mathbb{N}:x(n)\in [a,b),n\leq
N\})}{N}=b-a.$$ Now we want to generalize the concept of
distribution to cover both the uniform and ununiform cases.

\section{main results}
\begin{definition}\label{properly distributed}
A sequence $x:=\{x(n)\}^{\infty}_{n=1}\in l^{\infty}$ is called
properly distributed if for any Borel subset $S$ of
$[-\|x\|_{\infty},\|x\|_{\infty}]$ it holds that
\begin{align*}
w(x,S)&=\liminf_{n\rightarrow \infty}\frac{A(\{k:x(k+i)\in
S,k=0,1,\ldots,n-1\})}{n}\\
&=\limsup_{n\rightarrow \infty}\frac{A(\{k:x(k+i)\in
S,k=0,1,\ldots,n-1\})}{n}
\end{align*}  exists uniformly in $i\in \mathbb{N}$
and $w(x,S)$ is called the weight of $x$ with respect to $S$.
\end{definition}

If we treat a properly distributed sequence $x$ as a function
defined on $\mathbb{N}$, $x$ is analogous to the measurable function
in real analysis, with $w(x,S)$ corresponding to some measure
$\mu(\{n:x(n)\in S\})$ over $\mathbb{N}$. Though $w(x,S)$ indeed has
some similar behavior as a measure like nonnegativity and finite
additivity, $w(x,S)$ is not a measure in general setting, for it
fails to satisfy countable additivity. Here is an illustrating
example:

\begin{example}\label{not measure}
Let
$s_1=\{\underbrace{1,\ldots,1}_{n-times},\underbrace{0,0,\ldots}_{otherwise}\}$,
which is obviously properly distributed. If there exists a measure
$\mu$ over $\mathbb{N}$ such that $\mu(\{n:x(n)\in S\})=w(x,S)$ for
any properly distributed sequence $x\in l^{\infty}$ and Borel subset
$S$, then
$\mu(\{1,2,\ldots,n\})=w(x,[1-\varepsilon,1+\varepsilon))=0$, where
$\varepsilon$ is a sufficiently small positive number. Similarly, it
can further be implied that for any finite subset $E$ of
$\mathbb{N}$ it always holds $\mu(E)=0$. Since $\mu$ is countably
additive and $\mathbb{N}$ is the union of pairwise disjoint finite
subsets, it follows that $\mu(\mathbb{N})=0$. However, if we set
$s_2=\{1,\ldots,1,\ldots\}$, then $s_2$ is properly distributed and
$\mu(\mathbb{N})=w(s_2,[1-\varepsilon,1+\varepsilon))=1$, which
leads to a contradiction. Thus, such measure $\mu$ over $\mathbb{N}$
doesn't exist.
\end{example}

From Example \ref{not measure}, you may have already realized that
$s_1$ and $s_2$ represent a simple but useful class of properly
distributed sequences. Hence, we naturally give the following
definition of \emph{simply distributed sequences}, which would play
the similar role as ``simple functions'' in real analysis.

\begin{definition}\label{simply distributed}
A sequence $s:=\{s(n)\}^{\infty}_{n=1}\in l^{\infty}$ is called
simply distributed if $s$ is finitely-valued with range $\{a_1,
\ldots, a_m\}$ and it holds that
\begin{align*}
w(s,a_j)&=\liminf_{n\rightarrow\infty}\frac{A(\{k:s(k+i)=a_j,k=,0,1,\ldots,n-1\})}{n}\\
&=\limsup_{n\rightarrow\infty}\frac{A(\{k:s(k+i)=a_j,k=0,1,\ldots,n-1\})}{n}
\end{align*} exists
uniformly in $i\in \mathbb{N}$ for $j=1,\ldots,m$ and $w(s,a_j)$ is
called the weight of $s$ with respect to $a_j$.
\end{definition}

Though we cannot bring our work into the framework of measure and
integration(In fact, we really tried to do so at the beginning of
our research.), we still find much common feature between them,
which suggests us to generalize the measure-integration procedure in
real analysis to obtain a \emph{formal} integral to express the
unique Banach limit of almost convergent sequence. This would
partially answer the open question of \cite{Feng2}.

\begin{theorem}\label{integral 1}
If $s\in l^{\infty}$ is a simply distributed sequence with finite
range $\{a_1, \ldots, a_m\}$, then it is almost convergent with the
unique Banach limit $L(s)=\sum_{j=1}^m a_jw(s,a_j)$.
\end{theorem}

\begin{proof}
Let $S(s)$ denote the set of all sub-limits of $s$. Since $s$ is
finitely-valued, we have $S(s)\subseteq\{a_1, \ldots, a_m\}$ is
finite. Moreover, if $a_j\notin S(s)$, then $a_j$ must appear finite
times in $s$ with $w(s,a_j)=0$. Hence, by Theorem 4 of \cite{Feng1},
it implies that $s$ is almost convergent and for any Banach limit
$L$, $L(s)=\sum_{j=1}^m a_jw(s,a_j)$.
\end{proof}

From Theorem \ref{integral 1}, we can see that for any simply
distributed sequence $s$, its unique Banach limit could be expressed
as formal integral $L(s)=\sum_{j=1}^m a_jw(s,a_j)$. Then it
naturally arises the question that whether it is still true for
general properly distributed sequences. To this end, we'd like to
generalize the procedure of integration in real analysis. Firstly,
let us approximate properly distributed sequences by simply
distributed sequences.

\begin{lemma}\label{approximation}
For any properly distributed element $x\in l^{\infty}$, there is a
sequence of simply distributed elements
$\{s_k\}^{\infty}_{k=1}\subseteq l^{\infty}$ such that
$\lim_{k\rightarrow \infty} s_k=x$ under the norm
$\|\cdot\|_{\infty}$ in $l^{\infty}$.
\end{lemma}

\begin{proof}
For $k\in \mathbb{N}$, there is a partition
$$T_k:-\|x\|_{\infty}=a_0<\ldots<a_{m_k}=\|x\|_{\infty}$$ of
$[-\|x\|_{\infty},\|x\|_{\infty}]$ such that $\|T_k\|<1/k$. Define
\begin{equation*}
s_k(n)=
\begin{cases}
a_0, & \text{if $a_0\leq x(n)< a_1$},\\
\cdots & \cdots,\\
a_{m_k-1}, & \text{if $a_{m_k-1}\leq x(n)< a_{m_k}$.}
\end{cases}
n=1,2,3,\ldots
\end{equation*}
Since $x$ is properly distributed, it follows easily that each $s_k$
is simply distributed. According to the above construction, it is
obvious that $\|s_k-x\|_{\infty}<1/k$. Thus
$\lim_{k\rightarrow\infty} s_k=x$.
\end{proof}

\begin{theorem}
If $x\in l^{\infty}$ is any properly distributed sequence, then $x$
is almost convergent. And if $\{s_k\}_{k=1}^{\infty}$ is any
sequence of simply distributed sequences convergent to $x$ under the
$\|\cdot\|_{\infty}$ norm, for any Banach limit $L$, it always holds
that $\lim_{k\rightarrow\infty}L(s_k)=L(x)$.
\end{theorem}

\begin{proof}
For any Banach limit $L$, since $L$ is a bounded linear functional
on $l^{\infty}$ and $\lim_{k\rightarrow\infty}s_k=x$, it follows
that $\lim_{k\rightarrow\infty}L(s_k)=L(x)$. By Theorem
\ref{integral 1}, the value of each $L(s_k)$ is independent of $L$,
thus so is $L(x)$. We conclude that  $x$ is almost convergent and
the unique Banach limit is $\lim_{k\rightarrow\infty}L(s_k)$.
\end{proof}

Now we want to use the new theory to review the work of
\cite{Feng1}, which will be shown to be a special case in our
framework.

\begin{lemma}\label{two weights}
Let $x:=\{x(n)\}_{n=1}^{\infty}\in l^{\infty}$. Suppose $a$ is an
isolated sub-limit of $x$, and there exists $\varepsilon_0>0$ such
that $(a-\varepsilon_0,a+\varepsilon_0)\bigcap S(x)=\{a\}$. Then for
any $0<\varepsilon\leq\varepsilon_0$,
$w(x,[a-\varepsilon,a+\varepsilon))$ exists if and only if $w(a)$
does. Moreover, if they both exist, they are equal.
\end{lemma}

\begin{proof}
Like Proposition \ref{correction}, for any
$0<\varepsilon\leq\varepsilon_0$, let $\{x(n_k)\}$ denote all the
terms of $x$ that lying in $[a-\varepsilon,a+\varepsilon)$. Then,
similarly, it is easy to show that $\{x(n_k)\}$ is an essential
subsequence of $a$. And, for any $n,i\in \mathbb{N}$, we have
\begin{align*}
&\frac{A(\{j:a-\varepsilon\leq
x(i+j)<a+\varepsilon,j=0,1,\ldots,n-1\})}{n}\\
=&\frac{A(\{k:i\leq n_k \leq i+n-1\})}{n}.
\end{align*}
Consequently,
\begin{align*}
&\limsup_{n\rightarrow\infty}\frac{A(\{j:a-\varepsilon\leq
x(i+j)<a+\varepsilon,j=0,1,\ldots,n-1\})}{n}\\
=&\limsup_{n\rightarrow\infty}\frac{A(\{k:i\leq n_k \leq
i+n-1\})}{n},
\end{align*}
and
\begin{align*}
&\liminf_{n\rightarrow\infty}\frac{A(\{j:a-\varepsilon\leq
x(i+j)<a+\varepsilon,j=0,1,\ldots,n-1\})}{n}\\
=&\liminf_{n\rightarrow\infty}\frac{A(\{k:i\leq n_k \leq
i+n-1\})}{n}.
\end{align*}
Now it is clear that $w(x,[a-\varepsilon,a+\varepsilon))$ exists if
and only if $w(a)$ does. And, if they both exist, they are equal.
\end{proof}

Now it's time to include Theorem 4(\cite{Feng1}) into our framework.
\begin{theorem}
Suppose $x:=\{x(n)\}_{n=1}^{\infty}\in l^{\infty}$ and
$S(x)=\{a_1,a_2,\ldots,a_m\}$ is a finite set, where $a_i\neq a_j$
if $i\neq j$. If $w(a_j)$ exists for each $j$, then $x$ is properly
distributed.
\end{theorem}

\begin{proof}
For any interval $[c,d)$, if $[c,d)\bigcap
\{a_1,a_2,\ldots,a_m\}=\emptyset$, there would be at most finite
terms in $[c,d)$, so
\begin{align*}
w(x,[c,d))&=\liminf_{n\rightarrow \infty}\frac{A(\{k:x(k+i)\in
[c,d),k=0,1,\ldots,n-1\})}{n}\\
&=\limsup_{n\rightarrow \infty}\frac{A(\{k:x(k+i)\in [c,
d),k=0,1,\ldots,n-1\})}{n}\\
&=0
\end{align*}
exists uniformly in $i\in \mathbb{N}$. Otherwise, there are some
$a_j$s in $[c,d)$. Without loss of generality, we can assume only
$a_j$ lying $[c,d)$. In fact, if there are more than one such $a_j$,
we can decompose $[c,d)$ into disjoint subintervals such that each
contains only one $a_j$. From Lemma \ref{two weights}, since $w(a_j
)$ exists, we also have $w(x,[c,d))$ exists, and
$w(x,[c,d))=w(a_j)$. Thus we have proved that $x$ is properly
distributed.
\end{proof}

Moreover, we can reobtain the unique Banach limit of $x$ above,
using the approximation method by simply distributed sequences.

\begin{cor}
Suppose $x:=\{x(n)\}_{n=1}^{\infty}\in l^{\infty}$ and
$S(x)=\{a_1,a_2,\ldots,a_m\}$ is a finite set, where $a_i\neq a_j$
if $i\neq j$. If $w(a_j)$ exists for each $j$, then $x$ is almost
convergent, with the unique Banach limit $L(x)=\sum_{j=1}^m
a_jw(a_j)$ for any Banach limit $L$.
\end{cor}

\begin{proof}
For any sufficiently big $k\in \mathbb{N}$, define
\begin{equation*}
s_k(n)=
\begin{cases}
a_j, & \text{if $a_j-1/k\leq x(n)< a_j+1/k$,}\\
x(n), & \text{otherwise.}
\end{cases}
j=1,\ldots,m; n\in \mathbb{N}.
\end{equation*}
It is easy to see that each $s_k$ is a simply distributed sequence
with only $w(s_k,a_j)\neq 0$, and $L(s_k)=\sum_{j=1}^m
a_jw(s_k,[a_j-1/k,a_j+1/k))=\sum_{j=1}^m a_jw(a_j)$. From the
construction of $\{s_k\}_{k=1}^{\infty}$,
$\lim_{k\rightarrow\infty}s_k=x$ under the $\|\cdot\|_{\infty}$
norm. Then it follows that
$L(x)=\lim_{k\rightarrow\infty}L(s_k)=\sum_{j=1}^m a_jw(a_j)$.
\end{proof}

In Theorem 5 and 6 of \cite{Feng1}, sequences whose sub-limit sets
have limit points are considered. In order to keep the form
$L(x)=\sum_{a\in S(x)}aw(a)$, the authors made a great effort to
give a complex definition for the weight of limit points of $S(x)$.
Now, from our distribution viewpoint, it is very easy to understand
those complex formulae. Let us take Theorem 5 \cite{Feng1} for
example, Theorem 6 \cite{Feng1} is treated in a similar way locally
at each limit point of $S(x)$.

\begin{theorem}
Suppose $x:=\{x(n)\}_{n=1}^{\infty}\in l^{\infty}$ and $S(x)$ is
infinite but countable and has a unique limit point $p$, that is
$S'(x)=\{p\}$. If, furthermore, $w(a)$ exists for all $a\in S(x)$
and $a\neq p$, then $x$ is properly distributed, and for any Banach
limit $L$, $L(x)=\sum_{a\in S(x)}aw(a)$, where $w(p)=1-\sum_{p\neq a
\in S(x)}w(a)$.
\end{theorem}

\begin{proof}
For any sufficiently big $k\in \mathbb{N}$, define
\begin{equation*}
s_k(n)=
\begin{cases}
p, &  \text{if $p-1/k\leq x(n)\leq p+1/k$,}\\
a_j, & \text{if $a_j-1/k\leq x(n)< a_j+1/k$, and $a_j\notin [p-1/k,p+1/k)$,}\\
x(n), & \text{otherwise.}
\end{cases}
\end{equation*}

Since there are only finite $a_j\notin [p-1/k,p+1/k)$, each $s_k$ is
properly distributed and $\lim_{k\rightarrow\infty}s_k=x$. Moreover,
from Lemma \ref{two weights}, we have
$$L(s_k)=\sum_{a_j\notin
[p-1/k,p+1/k)}a_jw(a_j)+p(1-\sum_{a_j\notin [p-1/k,p+1/k)}w(a_j)).$$
Let $k\rightarrow\infty$, it follows that
$L(x)=\lim_{k\rightarrow\infty}L(s_k)=\sum_{a\in S(x)}aw(a)$, where
$w(p)=1-\sum_{p\neq a \in S(x)}w(a)$.
\end{proof}

\bibliographystyle{amsplain}

\end{document}